\documentclass[preprint, authoryear]{elsarticle}
\usepackage{amsmath, amsfonts, amssymb, amstext, amscd}

\newtheorem{thm}{Theorem}[section]
\newtheorem{prop}{Proposition}[section]
\newtheorem{lem}{Lemma}[section]
\newtheorem{cor}{Corollary}[section]

\newdefinition{alg}{Algorithm}[section]

\newproof{proof}{Proof}

\newcommand{\fq}{\mathfrak q}
\DeclareMathOperator{\Spec}{{Spec}}

\begin{document}

%----------------------------------------------------------------------------

\title{Polynomial Bounds for Invariant Functions Separating Orbits}
\author{Harlan Kadish\fnref{fn1}}
\ead{hmkadish@umich.edu}
\ead[url]{http://www-personal.umich.edu/~hmkadish/index.html}
\address{Department of Mathematics, University of Michigan,
2074 East Hall, 530 Church Street, Ann Arbor, MI 48109, USA } 
\fntext[fn1]{Supported by DMS-0502170, Enhancing the Mathematical Sciences Workforce in the 21st Century (EMSW21) Research Training Group (RTG):  Enhancing the Research Workforce in Algebraic Geometry and Its Boundaries in the Twenty-First Century}

\begin{abstract}
Consider the representations of an algebraic group $G$.  In general, polynomial invariant functions may fail to separate orbits.  The invariant subring may not be finitely generated, or the number and complexity of the generators may grow rapidly with the size of the representation.  We instead study ``constructible" functions defined by straight line programs in the polynomial ring, with a new ``quasi-inverse" that computes the inverse of a function where defined.  We write straight line programs defining constructible functions that separate the orbits of $G$.  The number of these programs and their length have polynomial bounds in the parameters of the representation.
\end{abstract}


\maketitle
\tableofcontents

%----------------------------------------------------------------------------

\section{Introduction}
\subsection{Background}
When an algebraic group $G$ acts on an affine variety $V$ over a field $k$, the \emph{orbit} of $x \in V$ is the set
\[ G\cdot x = \{g\cdot x \mid \forall g \in G\}.\]

\noindent Applications of invariant theory, such as computer vision, dynamical systems, and structural chemistry, demand constructive and more efficient techniques to distinguish the orbits of a group action.  When the group acts rationally, the invariant ring
\[ k[V]^G = \{ f \in k[V] \mid f(g \cdot x) = f(x) \ \forall g \in G\}\]

\noindent contains a finitely generated subalgebra $S$ with the following property: Let $p, q \in V$ have disjoint orbit closures, and suppose there exists $f \in k[V]^G$ such that $f(p) \neq f(q)$.  Then there exists $h \in S$ such that $h(p) \neq h(q)$ \citep{DK}.  We say that the function $h$ (and the algebra $S$) \emph{separates} the orbit closures of $p$ and $q$. Note that the functions in $S$, called \emph{separating invariants}, separate as many orbits as does $k[V]^G$.  Since $G$ is an algebraic group, $\overline{G\cdot p} = \overline{G\cdot q}$ implies $G\cdot p = G\cdot q$, because orbits are open in their closures.

This separating subalgebra $S$ has several weaknesses.  For one, existence proofs for $S$ may not be constructive for all algebraic groups: Kemper's algorithm to construct $S$ assumes a reductive group \citep{KemperComp}.  Even in the constructive case, although polynomial bounds exist for the degrees of generators for $k[V]^G$ under the action of a linearly reductive group \citep{PolyBounds}, construction algorithms for separating invariants do not, for general $G$, provide good bounds on the size of a separating subset, the degrees of its elements, or the complexity of its computation.  Kemper's algorithm, for example, requires two Gr\"{o}bner basis calculations, a normalization algorithm, and an inseparable closure algorithm.  Domokos used polarization to cut down the number of variables needed in separating invariants of reducible representations \citep{Domokos}, while Kemper provided new bounds, when $G$ is finite, on the required number of separating invariants \citep{KemperSep}.

As a more serious limitation, the invariant ring $k[V]^G$, and hence any subalgebra, may fail to separate orbit closures.  Even when $G$ is reductive, the polynomials in $k[V]^G$ can separate $G\cdot p$ and $G\cdot q$ if and only if $\overline{G\cdot p}\cap\overline{G\cdot q} = \emptyset$.  For example, when the multiplicative group $G=k^*$ acts on $\mathbb A^2$ by scaling points, one finds $k[x,y]^G = k$.

%----------------------------------------------------------------------------

\subsection{Separating Orbits with Constructible Functions}
To overcome the limitations of the invariant ring, we expand the set of regular functions on a variety to include a ``quasi-inverse" $\{f\}$ of a regular function $f$:
\[ \{f\}(p) = \begin{cases} 1/f(p) & f(p) \neq 0 \\ 0 & f(p) = 0 \end{cases}.\]

\noindent For $R = k[V]$, let $\widehat{R}$ denote the ring of ``constructible" functions $V \to k$ obtained by defining the quasi-inverse on $R$.  For example, if $f, g \in R$, then $\{\{f\} + g\} \in \widehat{R}$.  In fact, one can show that for any $h \in \widehat{R}$, there exists finitely many locally closed sets $E_i\subseteq V$ and $f_i$ regular on $E_i$ such that
\[ h = \sum_{i=1}^k f_i \cdot \chi_{E_i}\]

\noindent where $\chi_{E_i}$ is the characteristic function of a constructible subset $E_i \subseteq \Spec R$.

For a given group action, we seek to write down a finite set $\mathcal C$ of invariant, constructible functions that separate orbits.  That is, if $p, q$ lie in different orbits, then some function $f \in \mathcal C$ has $f(p) \neq f(q)$.  Even better, we would like the evaluation of $f$ at $p$ to be reasonably simple, in the sense of its complexity as a straight line program whose inputs are the coordinates of $p$.

Over an algebraically closed field $k$, fix an embedding of an $m$-dimensional algebraic group $G \hookrightarrow \mathbb A^\ell$.  Let  $R = k[x_1, \ldots, x_n]$, let $\rho: G \hookrightarrow GL_n(k)$ be a representation, let $r$ be the maximal dimension of an orbit, and let $N = \text{max}\{\text{deg}(\rho_{ij})\}$ be the degree of the representation.

\begin{thm}
There is an algorithm to produce a finite set $\mathcal C \subset \widehat{R}$ of invariant, constructible functions with the following properties:
\begin{enumerate}
\item The set $\mathcal C$ separates orbits.
\item The size of $\mathcal C$ grows as $O(n^2 N^{(\ell+m+1)(r+1)})$.
\item The $f \in \mathcal C$ can be written as straight line programs, such that the sum of their lengths is $O(n^3 N^{3\ell(r+1)+r})$.
\end{enumerate}
Hence the problem of deciding if two points lie in the same orbit can be solved with a polynomial number of algebraic operations in the coordinates of the points.
\end{thm}

More explicitly, for $p \in \mathbb A^n$ consider the orbit map $\sigma_p \colon G \to \mathbb A^n$ defined by $\sigma_p\colon g \mapsto g \cdot p$.  Note that $\overline{G \cdot p}$ is defined by the polynomials in the kernel of $\sigma_p^*\colon k[x_1, \ldots, x_n] \to k[G]$.  These polynomials amount to algebraic relations on the images $\sigma_p^*(x_1), \ldots, \sigma_p^*(x_n)$ in $k[G]$.  One can find all such relations up to some degree $d$ by Gaussian elimination.  The coefficients of these relations vary with $p$, but they cannot in general be written as regular functions of $p$. We may nevertheless write them with constructible functions, especially utilizing the fact that $f\{f\}(p)=1$ if $f(p)\neq 0$.  These constructible functions form the set $\mathcal C$.  Essentially, the idempotent constructible functions encode if-then branching into the formulas for our relations.

We proceed in four parts.  First, given a matrix $X$ encoding products of the $\sigma^*_p(x_i)$ and encoding $\mathbb I(G)$ up to some degree $d$, we produce a matrix of constructible functions that gives the entries of the reduced row echelon form of $X$, as functions of $p$.  From these entries follow formulas for the kernel vectors of $X$ and hence relations on the $\sigma^*_p(x_i)$.  We next establish a degree bound for the relations sufficient to generate the ideal $\fq$ with $\mathbb V(\fq) = \overline{G \cdot p}$.  By considering a generating set for $\fq$, we provide an algorithm that produces straight line programs for the $G$-invariant functions in the set $\mathcal C$.  We show that these straight-line programs separate orbits and have polynomial length, and we establish polynomial bounds for their number in terms of $n$ and the degree
$N$ of the representation.

%----------------------------------------------------------------------------

\section{Formulas for Reduced Row Echelon Form}

\subsection{Straight Line Programs}

Let $V$ be a set, $F$ a field, and let $R$ be an $F$-subalgebra of the $F$-valued functions on $V$.  Let $A = (a_{-m}, \ldots, a_{-1}) \in \widehat{R}^m$ be a finite, ordered subset of $\widehat{R}$.  Consider a tape of cells with $a_i \in A$ in position $i$.  A straight line program $\Gamma$ is a finite, ordered list of instructions $\Gamma = (\Gamma_0, \ldots, \Gamma_{\ell-1})$.  Each instruction $\Gamma_i$ is of the form $(\star; j,k)$ or $(\star; j)$, where $\star$ is an operation and $j,k$ are positive integers referring to tape entries in positions $i-j$ and $i-k$, that is, $j$ and $k$ cells before $i$, respectively.  The length $\ell = \vert \Gamma \vert$ measures the complexity of the computation.

To execute $\Gamma$ on input $A$, for $i = 0, \ldots, \ell-1$ write $a_i$ in tape position $i$ as follows:
\[ a_i = \begin{cases}
a_{i-j} + a_{i-k} & \text{if }\Gamma_i = (+;j,k)\\
a_{i-j} - a_{i-k} & \text{if }\Gamma_i = (-;j,k)\\
a_{i-j}\cdot a_{i-k} & \text{if }\Gamma_i = (\times; j,k)\\
\{a_{i-j}\} & \text{if }\Gamma_i = (\{\cdot\}; j)\\
c & \text{if }\Gamma_i = (\textup{const}; c) \text{ for } c \in F \\
a_{i-j} & \text{if }\Gamma_i = (\textup{recall}; j)
\end{cases} \quad \text{where $j,k < i.$} \] 

\noindent The ``recall" instruction of position $j$ serves to collect relevant computations at the end of the tape.  Define the order-$d$ output of $\Gamma$ by $\text{Out}_d(\Gamma, A) = (a_{\ell-d}, \ldots, a_{\ell-1})\in \widehat{R}^d$, where $\ell = \vert \Gamma\vert$.  We omit the $d$ where convenient.  A straight line program hence defines a constructible function $\widehat{R}^m \to \widehat{R}^d$. 

Write $\Gamma^{(2)} \circ \Gamma^{(1)}$ for the composition of two straight line programs, in which the input of $\Gamma^{(2)}$ is $\text{Out}_d(\Gamma^{(1)}, A)$ for some $d$ depending on $\Gamma^{(2)}$.  Then $\Gamma^{(2)} \circ \Gamma^{(1)}$ has input $A$, and we execute $\Gamma^{(2)} \circ \Gamma^{(1)}$ by concatenating the instruction lists.  For a detailed treatment, see \citet{Complexity}.

%----------------------------------------------------------------------------

\subsection{Outline of the Algorithm}

Let $A = (a_{ij})$ be an $m \times n$ matrix over a field $k$.  Define the \emph{triangular reduced row echelon form} (tRREF) of $A$ to be the $n \times n$ matrix $R_A=(r_{ij})$ whose $j$th row $\bf{r_j}$ is nonzero if and only if the reduced row echelon form (RREF) of $A$ has a pivot in column $j$.  In that case, $\bf{r_j}$ is the row of the RREF of $A$ containing that pivot.  For example,

\[ \text{RREF(A)} = \left(\begin{array}{ccc} 1 & 2 & 0 \\ 0 & 0 & 1 \\ 0 & 0 & 0 \\\end{array}\right) \quad \text{corresponds to} \quad  \text{tRREF}(A) = \left(\begin{array}{ccc} 1 & 2 & 0 \\ 0 & 0 & 0 \\ 0 & 0 & 1 \\\end{array}\right). \]

\noindent This new form simplifies the identification of pivots: the (usual) RREF of $A$ has a pivot in column $j$ if and only if $r_{jj} = 1$ in the tRREF.

\begin{prop} Let $(a_{ij})$ be an $m\times n$ matrix with entries in any field $k$.  Then there exists a straight line program $\Gamma^{tR}$ of length $O(mn^2 + n^3)$ such that $\text{Out}_{n^2}(\Gamma^{tR}, (a_{ij}))$ are the entries of the triangular RREF of $(a_{ij})$. The program gives constructible functions for these entries in terms of the $a_{ij}$.
\end{prop}

The proposition does not require $k$ to be algebraically closed, but we will need this condition for the later geometric reasoning about orbits.  Note also that while the classical Gausssian elimination algorithm requires branching, the straight line program $\Gamma^{tR}$ simulates branching in the computation of the quasi-inverse.  The psuedo-code below proves the proposition in general terms; the subsections that follow provide specific constructions.

\begin{alg} Let $A = (a_{ij})$ be an $m\times n$ matrix.
\begin{enumerate}
\item For $i = 2, \ldots m$, if $a_{11} = 0$, exchange the first row of $A$ with the $i$th row.  After these steps, either $a_{11} \neq 0$, or $a_{i1} = 0$ for all $i$.
\item Multiply $a_{11}$ by $\{a_{11}\}$, and multiply the rest of the first row by \\ $(1 - a_{11}\{a_{11}\} + \{a_{11}\})$.  This is equivalent to dividing the first row by $a_{11}$ if $a_{11} \neq 0$.
\item For $i = 2, \ldots, m$, subtract $a_{i1} \cdot (a_{11}, \ldots, a_{1n})$ from row $i$.  As a result, $a_{i1} = 0$ for all $i \geq 2$.
\item Let $A' = (a_{ij})_{j\geq 2}$ and $A'' = (a_{ij})_{i,j\geq 2}$, as below:
\[
A = \left( \begin{array}{ccc} * &  &  \\ 0 & A' &  \\ \vdots & & \\ 0 & & \end{array} \right) = \left(\begin{array}{ccc} * & \cdots &  * \\ 0 & & \\ \vdots & A'' & \\ 0 & & \end{array} \right) \]

\noindent Let $A_0''$ be the $m \times (n-1)$ matrix formed by appending a row of zeros to the bottom of $A''$; then $A'$ and $A_0''$ have the same dimensions.
\item Define $B = (1-a_{11})\cdot A' + a_{11}\cdot A_0''$.
\item Recursively compute the tRREF of $B$; call it $R_B$, an $(n-1) \times (n-1)$ matrix.
\item Let $R_A$ be the $n \times n$ matrix below:
\[ R_A = \left( \begin{array}{ccc} a_{11} & \cdots & a_{1n} \\ 0 & R_B & \\ 0 & & \end{array}\right).\]
\item Let $\bf{r_k}$ be the $k$th row of $R_A = (r_{ij})$.  For $k = 2, \ldots, n$, subtract $a_{1k} \cdot \bf{r_k}$ from the first row of $R_A$.  This reduction produces the triangular RREF of $A$.
\end{enumerate}
\end{alg}

The following formulas specify straight line programs for the entries of the triangular RREF matrix $R_A$, and hence define $\Gamma^{tR}$.

%----------------------------------------------------------------------------

\subsection{Formulas for Gaussian Elimination}

Recall that the first step of the algorithm exchanges the first row of $(a_{ij})$ with the $i$th row if $a_{11} = 0$, for $i = 2, \ldots, m$.  Hence for an $m\times n$ input matrix $X$, this step requires $m-1$ programs $E_i$ such that $Y = \text{Out}_{mn}(E_i, X)$ flips the first and $i$th rows if necessary.  The following formulas describe the entries of $Y = (y_{ij})$:

\begin{eqnarray}
y_{11} &=& x_{11} + (1-x_{11}\{x_{11}\})x_{i1} \nonumber\\
y_{1j} &=& x_{1j} + (1-x_{11}\{x_{11}\})\cdot (x_{ij} - x_{1j}) \textup{ for all $j>1$} \nonumber\\
y_{i1} &=& x_{i1}\cdot x_{11}\{x_{11}\} \nonumber\\
y_{ij} &=& x_{1j} + x_{11}\{x_{11}\} \cdot (x_{ij} - x_{1j}) \textup{ for all $j>1$} \nonumber\\
y_{kj} &=& x_{kj} \textup{ for all $k\neq 1, i$, and for all $j$.} \nonumber
\end{eqnarray}

\noindent For example, the straight line program for $y_{11}$ in $E_i$ takes inputs $x_{11}$ in position -2 and $x_{i1}$ in position -1, and then performs the following steps:
\begin{enumerate}
\item[(0)] $( \{\cdot\}; 2)$
\item[(1)] $(\times; 3, 1)$
\item[(2)] $(\text{const}; 1)$
\item[(3)] $(-; 1,2)$
\item[(4)] $(\times; 1, 5)$
\item[(5)] $(+; 7, 1)$
\end{enumerate}

\noindent The formulas for the other $y_{ij}$ have similarly obvious representations as straight line programs.  If we concatenate these programs within $E_i$, so that all the entries of $Y$ appear in various (known!) positions on the tape, then we can save the recall steps for the end, and we need only compute $\{x_{11} \}$, $x_{11}\{x_{11}\}$, $(1- x_{11}\{x_{11}\})$, and $(x_{ij} - x_{1j})$ once.  With these efficiencies, the program $E_i$ introduces 1 quasi-inverse, 1 call to $k$, $3n$ additions, and $2n$ multiplications.  Thus the concatenation of $E_2, \ldots, E_{m-1}$ requires $2n(m-1)$ multiplications, $3n(m-1)$ additions, $n-1$ calls to $k$, $n-1$ quasi-inverses, and $mn$ recalls to collect the entries of $Y$ in the last $mn$ cells of the tape.  Call this concatenation $\Gamma^{E}$; we will use it later to collect nonzero rows of a matrix.

Step (2) of the algorithm requires 1 quasi-inverse, 1 subtraction, 1 addition, $n$ multiplications, and $n$ recalls.

These next formulas perform step (3), on an $m\times n$ input matrix $(x_{ij})$:
\begin{eqnarray}
y_{i1} &=& 0 \textup{ for all $i>1$} \nonumber\\
y_{ij} &=& x_{ij} - x_{1j}\cdot x_{i1} \cdot x_{11}\{x_{11}\} \textup{ for all $i,j>1$.}\nonumber
\end{eqnarray}

\noindent These programs require $(m-1)(n-1)$ additions, $(n-1)(m-1)$ multiplications, and $mn$ recalls.  Step (5) next requires 1 subtraction, $m(n-1)$ additions, $2m(n-1)$ multiplications, and $m(n-1)$ recalls.

To perform the reductions in step (8), consider the following formula for $r_{1j}$, where $j \geq 2$:
\begin{eqnarray*}
r_{1j} := (1-r_{jj}) \cdot \left( r_{ij}  + (\right. &-& r_{22}\cdot r_{12}r_{2,j} \\
					&-& r_{33}\cdot r_{13}r_{i3,j}  \\
					&-& \cdots \\
					&-& \left. r_{j-1,j-1}\cdot r_{1, j-1}r_{j-1,j})\right),
\end{eqnarray*}

\noindent This formula sets $r_{1j} = 0$ if there is a pivot in column $j$, that is, if $r_{jj}=1$.  Otherwise, the formula subtracts from $r_{1j}$ the effects of clearing columns $<j$.  The reduction of $r_{1j}$ requires 1 call to $k$, $j$ additions/subtractions, $2(j-2) +1$ multiplications (since $j \geq 2$), and $n^2$ recalls, so reducing the first row has total complexity $O(n^2)$.

The above formulas specify a straight line program $\Gamma^{tR}$ such that $\text{Out}_{n^2}(\Gamma^{tR}, A)$ are the entries of the tRREF of $A$.  Counting the necessary operations yields asymptotic total complexity estimates for the programs.   The recursion on an $m \times t$ matrix has total complexity $O(mt + t^2)$.  Summing $t$ from 1 to $n$ yields total complexity $O(mn^2 + n^3)$.
%----------------------------------------------------------------------------

\subsection{Collecting Nonzero Rows}

Lastly, the main algorithm that computes orbit closures requires a program $\Sigma$ that, given an indicator vector $v$ of 0s and 1s, collects the rows $i$ of a matrix such that the $i$th entry of $v$ is 1.  For example, the diagonal of $R_A$ indicates the nonzero rows of $R_A$.  Given $R_A$ and its diagonal as input, the program $\Sigma$ would output an $n\times n$ matrix whose first $\text{rank}(A)$ rows include the traditional RREF of $A$.  We will never need to compute the traditional RREF in practice, because the main algorithm runs more efficiently using $R_A$.

Recall the algorithm $\Gamma^E$ that exchanges the first row of a matrix $X$ with subsequent rows until the output has $y_{11} \neq 0$, if possible.  Define $\Sigma$ as follows: for an $m \times n$ input matrix $X$ and an indicator $m$-vector $v$, form a new matrix $X'$ by adjoining $v$ as a column to the left side of $X$:

\[ X' = \left( \begin{array}{cccc} v_1 & x_{11} & \cdots & x_{1n}\\ v_2 & x_{21} & \cdots & x_{2n} \\ \vdots & \vdots & \vdots & \vdots \\ v_{m} & x_{m1} & \cdots & x_{mn} \\ \end{array} \right).\]

\noindent After applying $\Gamma^E$ to $X'$, the first row of $X'$ with $v_i \neq 0$ becomes the first row of the output $Y=(y_{ij})$.  Record $\mathbf{r_1} := (y_{12}, \ldots, y_{1,n+1})$ and apply $\Gamma^E$ to the last $m-1$ rows of this $Y$.  Let $\Sigma$ denote this series of $m$ recursions of $\Gamma^E$.  Since $\Gamma^E$ applied to an $s\times (n+1)$ matrix has total complexity $O(sn)$, the procedure $\Sigma$ has complexity $O(m^2n)$. Concatenating $\Sigma$ with the straight line program for the tRREF yields the following:

\begin{cor} Let $(a_{ij})$ be an $m\times n$ matrix with entries in any field $k$.  Then there exists a straight line program of length $O(mn^2 +m^2n+ n^3)$ for the (classical) RREF $(r_{ij})$of $(a_{ij})$. The program gives constructible functions for $r_{ij}$ in terms of the $a_{ij}$.
\end{cor}

%----------------------------------------------------------------------------

\subsection{Computing Kernels of Linear Maps}

To compute the kernel up to degree $d$ of a $k$-algebra homomorphism, one can write the homomorphism on elements of degree $\leq d$ as a matrix in RREF.  Finding the kernel of a matrix $R$ in RREF is equivalent to solving the system of equations $R \cdot (x_1, \ldots, x_n)^T = 0$:  for every pivot $r_{ij}$, write an equation
\[ x_j = -r_{i,j+1}x_{j+1} - r_{i,j+2}x_{j+2} - \cdots - r_{i,n}x_{n}.\]

\noindent Set each free variable equal to 1 in turn, set the other free variables to 0, and read off the vector of values in the pivot variables.  These vectors give a basis for the kernel of $R$, hence of the original map.  The basis is canonical because the RREF is canonical.

To compute the kernel of an $m\times n$ matrix $A$, we use the $n\times n$ matrix $R_A$ containing the rows of the RREF of $A$: recall there is a pivot in the $j$th column of the RREF if and only if the row containing that pivot is $j$th row of $R_A = (r_{ij})$, if and only if $r_{jj}=1$.  Otherwise, $r_{jj} = 0$.

\begin{lem}
Let $R_A$ be the $n\times n$ tRREF of a matrix $A$.  Then there exists a straight line program $\Gamma^K$ of length $O(n^2)$ such that $Out_{n^2}(\Gamma^K, R_A)$ gives the kernel of $A$.
\end{lem}

\begin{proof}
Claim that the kernel of $A$, is given by the following vectors $\phi_1, \ldots, \phi_n$, in terms of $R_A =(r_{ij})$:
\[ \phi_j := (1-r_{jj}) \cdot (-r_{1j}, -r_{2j}, \ldots, \overbrace{1}^\text{$j$th place}, \ldots, -r_{nj}).\]

\noindent Indeed, recall that the kernel of a RREF matrix has one basis vector for each non-pivot column.  Namely, $\phi_j = 0$ if and only if column $j$ of the RREF has a pivot.  Otherwise, $\phi_j \neq 0$, as follows: Put the free variable $x_j := 1$.  Now, $r_{kj}=0$ unless there is a pivot in column $k$ of the RREF.  Set each pivot variable $x_{kk}$ equal to the negation of the $j$th entry of the row containing that pivot.

Of course, $r_{ij} = 0$ whenever $i>j$, but such simplifications complicate the formulas without improving the asymptotic complexity.  As written, each $\phi_j$ requres 2 calls to $k$, 1 addition, $n$ scalar multiplications, and $n$ other multiplications.  Upon adding recall instructions, computing the kernel has complexity $O(n^2)$.
\end{proof}

%----------------------------------------------------------------------------

\section{Degree Bounds for Orbit Closures}

We relate the degree of a variety to the degrees of polynomials that can define that variety.  By bounding the degree of an orbit closure $\overline{G \cdot p}$, we can bound the degree of the defining polynomials.

\begin{lem}
Let $V = \mathbb V(f_1, \ldots, f_r)$ have codimension $m$ in $\mathbb A^n$.  Then there exist $m$ generic linear combinations $g_i = \sum a_{ij} f_j$ such that
\[ W := \mathbb V(g_1,\ldots, g_m) \supseteq V\]

\noindent and $W$ has codimension $m$.
\end{lem}

\begin{proof}
Induct on the number $r$ of given defining equations for $V$. The case $r=1$, implying $m=1$, is clear.  Assume the lemma holds for a variety defined by $r-1$ equations, and consider $V' = \mathbb V(f_1, \ldots, f_{r-1})$.  If $V'$ still has codimension $m$, then the result follows by the induction hypothesis.  Otherwise, $V'$ has codimension $m-1$.  By the induction hypothesis, there exist $m-1$ generic linear combinations $g_i$ of $f_1, \ldots, f_{r-1}$ such that $W' = \mathbb V(g_1, \ldots, g_{m-1}) \supseteq V'$ and $W'$ has codimension $m-1$.

Since $W'$ is defined by $m-1$ equations, every component $Z_k$ of $Y'$ has codimension $m-1$.  It follows that on each $Z_k$, one of $f_1, \ldots, f_r$ is not identically zero.  So for each $Z_k$, and for every point $p \in Z_k$, we may consider the proper hyperplane $H_{k,p}\subset\mathbb A^r$ defined by the vanishing of
\[ x_1f_1(p) + x_2f_2(p) + \ldots + x_r f_r(p)\in k[x_1, \ldots, x_r].\]

\noindent Let $H_k = \cap_{p\in Z_k} H_{k,p}$.  Then $\cup_k H_k$ is a closed union of finitely many subspaces of $\mathbb A^r$.  Thus for any choice of $(a_1, \ldots, a_r)$ in the dense set $\mathbb A^r - \cup_k H_k$, the polynomial $g_m = \sum a_i f_i$ is not identically zero on any $Z_k$.  Therefore $Y = \mathbb V(g_1, \ldots, g_{m-1},g_m)$ contains $V$ and has codimension $m$.
\end{proof}

Let $V\subseteq \mathbb A^n$ be an equidimensional affine variety of codimension $m$.  Define the \emph{degree} of $V$ to be
\[ \deg(V) = \# H\cap V,\]

\noindent where $H$ is a generic linear subspace of dimension $m$.

\begin{prop}
Let $V\subseteq \mathbb A^n$ be a Zariski closed subset of degree $d$.  Then there exists an ideal $\fq$, generated by polynomials of degree $\leq d$, such that $\sqrt \fq = \mathbb I(V)$.  In particular, $\mathbb V(\fq) = V$.
\end{prop}

\begin{proof}
It suffices to find, for every point $p\not\in V$, a polynomial $f$ of degree $\leq d$ such that $f$ vanishes on $V$ but not at $p$.  If $V$ is a hypersurface, then $V = \mathbb V(f)$ with $\deg(f) = \deg(V)$, and we are done.  Otherwise, assume $V$ has codimension greater than 1.  Without loss of generality, further assume that $p$ is the origin.

To find a polynomial vanishing on $V$ but not at the origin, we project $V$ until an image has codimension 1.  Define $\pi \colon \mathbb A^n \to \mathbb P^{n-1}$ by $\pi\colon (x_1, \ldots, x_n) \mapsto [x_1 \colon \ldots \colon x_n]$.  Since $\dim \overline{\pi(V)} \leq \dim V<n-1$, there exists a point $[L] \in \mathbb P^{n-1} - \overline{\pi(V)}$.  Let $C(V) = \pi^{-1}(\overline{\pi(V)})$, the cone over $\overline{\pi(V)}$.  Then  $L = \overline{\pi^{-1}([L])}$ has $L\cap \overline{C(V)} = \{0\}$.

Assume without loss of generality that $L$ is the $x_n$-axis, and consider the projection $\phi\colon \mathbb A^n \to \mathbb A^{n-1}$ along $L$, defined by $\phi\colon (x_1,\ldots, x_n) \mapsto (x_1, \ldots, x_{n-1})$.  Because $\overline{C(V)}$ is a cone, the restriction of $\phi$ to $\overline{C(V)}$ is a finite map onto $\mathbb A^{n-1}$.  In particular, $\phi(V)$ is closed in $\mathbb A^{n-1}$.  Since $L$ is disjoint from $V$, $\phi(0) = 0$ remains outside the closed set $\phi(V)$.

Continue projecting until $\phi: \mathbb A^n \to \mathbb A^{n-m+1}$ gives $\phi(V)$ with codimension 1 (and dimension $\dim V$ after each projection).  Now, $\deg(\phi(V)) \leq d$.  Thus there exists a polynomial $f$ of degree $\leq d$ such that $f$ vanishes on $\phi(V)$ but $f(0)\neq 0$.  Hence $f \circ \phi(V) = 0$ but $f\circ \phi(0)\neq 0$.  As $\phi$ is defined by linear polynomials, the polynomial $f \circ \phi$ has degree $\leq d$, and the result follows.
\end{proof}

Now consider an algebraic group $G$ acting on affine $n$-space.  When we can bound the degree of an orbit closure $\overline{G\cdot x}$, then we can produce a degree bound for polynomials $f_i$ such that $\overline{G\cdot x} = \mathbb V(f_1, \ldots, f_r)$.  For an overview of bounds for the degrees of orbits and the (polynomial) degrees of generating invariants, see \cite{PolyBounds}.

\begin{prop}
Let $G$ be an algebraic group of dimension $m$, embedded in $\mathbb A^\ell$ with ideal $\mathbb I(G) = (h_1, \ldots, h_s)$.  Set $M = \max\{\deg(h_i)\}$.

Suppose $G$ acts on $\mathbb A^n$ with representation
\[ \rho\colon G \to GL_n \quad \text{defined by} \quad \rho\colon g \mapsto (\rho_{ij}(g)), \]

\noindent and set $N = \max\{\deg(\rho_{ij})\}$.  If $\overline{G\cdot x}$ is an orbit closure with dimension $r$, then
\[ \deg(\overline{G\cdot x}) \leq N^r M^{\ell-m}.\]
\end{prop}

\begin{proof}
Let $d = \deg(\overline{G\cdot x})$.  For a generic $(n-r)$-dimensional linear subspace $H\subseteq \mathbb A^n$, by definition $d = \#(\overline{G\cdot x} \cap H)$.  Let $\sigma\colon g \mapsto g\cdot x$ be the orbit map.  Then the degrees of the polynomials defining $\sigma$ are bounded by $N$.  Hence $\sigma^{-1}(H) = \mathbb V(u_1, \ldots, u_r) \subseteq G$ has $\deg(u_i) \leq N$ and has $\geq d$ irreducible components.

By the first lemma above, there exist generic linear combinations $f_j$ of the generators of $\mathbb I(G)$ such that $\mathbb V(f_1, \ldots, f_{\ell-m})$ is a complete intersection and contains $G$.  Thus
\[ \sigma^{-1}(H) \subseteq \mathbb V\left(u_1, \ldots, u_r, f_1, \ldots, f_{\ell-m}\right)\subset \mathbb A^\ell. \]

\noindent Consider the vanishing of the homogenized polynomials

\[ \mathbb V\left(\overline{u}_1, \ldots, \overline{u}_r, \overline{f}_1, \ldots, \overline{f}_{\ell-m}\right) \subset \mathbb P^\ell.\]

\noindent By a generalization of B\'ezout's theorem (see \citet{Fulton}, section 12.3.1), the number of irreducible components of this variety is (generously) bounded by
\[ \prod_i \deg(\mathbb V(\overline{u}_i)) \cdot \prod_j \deg(\mathbb V(\overline{f}_j)) =  \prod_i \deg(\overline{u}_i) \cdot \prod_j \deg(\overline{f}_j) \leq N^r M^{\ell-m}.\]

\noindent This number then also bounds $d$.
\end{proof}

\begin{cor}
With the hypotheses of the previous proposition, there exist polynomials $f_1, \ldots, f_t$ such that $\overline{G\cdot x} = \mathbb V(f_1, \ldots, f_t)$ and
\[ \deg(f_i) \leq \deg(\overline{G\cdot x}) \leq N^r M^{\ell-m}.\]
\end{cor}

%----------------------------------------------------------------------------

\section{Separating Orbits}

Let $\rho: G \hookrightarrow GL_n$ act on $\mathbb A^n$ as in Section 3.  For $p \in \mathbb A^n$, there exists an ideal $\fq$ such that $\mathbb V(\fq) = \overline{G\cdot p}$ and $\fq$ is generated in degree $\leq N^r M^{\ell-m}$.  We will establish straight line programs for the orbit-separating set $\mathcal C$ by considering a generating set for $\fq$.  We prove that these programs define invariant functions separating the orbits of $G$.  The length of these programs will be polynomial in the dimension $n$ and the degree $N$ of the representation.

\subsection{The Orbit Separating Algorithm}

Input the embedding of $G \hookrightarrow \mathbb A^\ell$ and the orbit map $\sigma_p \colon g \mapsto g\cdot p$ as above, which varies with $p$.  Let $k[x_1, \ldots, x_n]$ be the coordinate ring of $\mathbb A^n$.  Then $\ker \sigma_p^* = \mathbb I(G\cdot p)$, but to define $\overline{G\cdot p}$ it suffices to compute a $k$-basis for $\ker \sigma_p^*$ up to degree $N^r M^{\ell-m}$.  The elements of this $k$-basis generate $\fq$ as an ideal.

For each $i =1, \ldots, N^r M^{\ell-m}$, the following algorithm computes a canonical $k$-basis for $\ker \sigma_p^*$ in degree $\leq i$, but for each polynomial in the basis the algorithm only outputs constructible functions (of $p$) that give the monomial coefficients appearing in that basis, whatever the monomials may be. Hence the algorithm forgets the generating set of the ideal $\fq$.  This forgetting allows the algorithm to have polynomial length as a straight line program, since the number of possibly monomials grows exponentially with $n$.

In the most precise sense, given a point $p \in \mathbb A^n$, the following algorithm concatenates straight line programs to output a $G$-invariant vector $\mathcal C$ over $k$.  In fact, each entry of $\mathcal C$ is a straight line program in terms of the coordinates of $p$.  Thus the algorithm prescribes a vector $\mathcal C$ of $G$-invariant constructible functions that separate orbits: points in distinct orbits produce distinct vectors.  The proofs for the $G$-invariance and orbit separation will follow.  

Choose a monomial order for the monomials spanning $k[z_1,\ldots, z_\ell]$.  As a preliminary calculation, compute a Gr\"obner basis and a $k$-basis for $\mathbb I(G)$ up to degree $N^{r+1}M^{\ell-m}$.  Let $B(d)$ denote the set of elements of the $k$-basis up to degree $d$.  Also, for a vector $w$, let $\pi_t(w)$ denote the vector of the the first $t$ entries of $w$.

Lastly, since all computations occur in $k[G]$, we must predict the dimension of $k[G]_{\leq d}$.

\begin{lem}  Let $m = \dim G$.  There exists a function $H(d)$, computable from a Gr\"obner basis for $\mathbb I(G)$, such that $H(d) = \dim_k k[G]_{\leq d}$ for all $d \geq 0$, and $H(d) \leq O(d^m)$.
\end{lem}

\begin{proof}  Suppose $R=k[G]$ is generated as a $k$-algebra by $f_1, \ldots, f_r$ of degree 1.  Define $S = k[f_1t, \ldots, f_rt, t] \subseteq R[t]$, and claim $S_d = R_{\leq d}\cdot t^d$, where $S$ is graded by $t$-degree.  The inclusion $\supseteq$ is clear, and if $h \in S_d$ is a homogeneous polynomial in $t$, then the coefficients of $t^d$ can have $R$-degree no greater than $d$ (less, for example, in the term $f_1t\cdot t^{d-1}$).  Let $H(d)$ be the $d$th coefficient of the Hilbert series of $S$, which we may compute from a Gr\"obner basis for $\mathbb I(G)$.  Then $H(d) = \dim_k R_{\leq d}$.  Since $S$ has dimension bounded by $m+1$, the Hilbert polynomial for $S$ has degree bounded by $m$.  Thus $H(d) \leq O(d^m)$.
\end{proof}

\begin{alg}
\mbox{}
\begin{enumerate}
\item For $j = 1, \ldots, n$, let $v_j$ be the vector of coefficients of $\sigma^*_p(x_j)$ with respect to the (ordered) monomial basis of $k[z_1\ldots, z_\ell]$.
\item $V_1 := (v_1,\ldots, v_n)$.
\item $i:=1$, $\mathcal C_0 = \emptyset$.
\item Put the vectors of $V_i = (v_1, \ldots, v_{k_i})$, in order, in the first $k_i$ columns of a matrix $X_i$; fill subsequent columns with $B(iN)$.
\item Compute $\text{Out}(\Gamma^{tR}, X_i)$, the tRREF $X_i$.
\item Compute $\beta := \text{Out}(\Gamma^K, \text{Out}(\Gamma^{tR}, X_i))$, a basis for $\ker X_i$.
\item Let $\mathcal C_i := \mathcal C_{i-1} \cup \{\pi_{k_i}(v) \mid v \in \beta\}$.
\item IF $N^r M^{\ell-m} = i$, THEN output $\mathcal C = \mathcal C_i$, and STOP.
\item Let $Y$ be the matrix whose rows are the vectors in $V_i$.  Let $D$ be the first $k_i$ entries on the diagonal of the tRREF $X_i$.
\item Compute $Y':=\text{Out}(\Sigma, \{Y,D\})$, the rows of $Y$ indicated by $D$.
\item Let $L_i$ be the first $H(i)$ rows of $Y'$.
\item IF $k_i = \#(\text{rows of Y}) < H(i)$, THEN pad $L_i$ with zeros so that $L_i$ has precisely $H(i)$ vectors.
\item $V_{i+1} := L_i \cup \left( \{\sigma^*_p (x_1), \ldots, \sigma^*_p (x_n)\} \cdot \{v_j \in L_i \mid j > H(i-1). \} \right).$
\item $i :=i+1$.
\item GOTO (4).
\end{enumerate}
\end{alg}

The final steps of each iteration require some remarks.  For step (10), recall that the nonzero entries of the diagonal $D$ of the tRREF of $X_i$ indicate which columns of $X_i$ are linearly independent.  These are the image vectors the algorithm should preserve for the next iteration, so that it can proceed with a polynomial number of multiplications.  In step (13), we multiply the $\sigma_p^*(x_i)$ only by these newfound vectors.

Step (12) can be accomplished in the context of straight line programs because we can predict the iteration $i$ at which $k_i  \geq H(i)$ first occurs, independent of the choice of $p$.  At step (13) we multiply $L_i$ by all $\sigma^*_p(x_i)$ because, in principle, all $\sigma^*_p(x_i)$ could be linearly independent modulo $\mathbb I(G)$.  As $i$ increases, the vectors in each $V_i$ describe the images of larger monomials $x^I$, $I$ a multi-index, in $k[x_1,\ldots, x_n]$. The algorithm terminates when we have considered a $k$-basis for the polynomials of degree up to $N^r M^{\ell-m}$ that vanish on $\overline{G\cdot p}$.  By the previous section, the elements of that $k$-basis generate an ideal whose radical is $\mathbb I(\overline{G\cdot p})$.

\begin{prop}
The constructible functions defined by the set $\mathcal C$
\begin{enumerate}
\item are constant on the orbit of $p \in \mathbb A^n$, and hence invariant under the usual action $g\cdot f(x) = f(g^{-1}\cdot x)$ for $g \in G$,  
\item separate orbits.
\end{enumerate}
\end{prop}

\begin{proof}
To show that the functions defined by the straight line programs in $\mathcal C$ are invariant, choose $p \in \mathbb A^n$ and $q \in G\cdot p$.  Let $X_i(p)$ be the matrix produced in step (4) of the algorithm in the $i$th iteration. Let $X_i^V(p)$ be the first $\vert V_i\vert = k_i$ columns of $X_i(p)$, that is, those containing the vectors in $V_i(p)$.  Now, $X_1^V(p)$ and $X_1^V(q)$ have the same kernel, because (a) as maps $k[x_1, \ldots, x_n]_1 \to k[G]_{\leq N}$ they have the same basis $x_1,\ldots, x_n$ for their domain, and because (b) the kernel of each matrix must span $\mathbb I(G\cdot p)_1$.  Thus $X_1^V(g\cdot p)$ = $A \cdot X_1^V(p)$ for some matrix $A$.  In particular, $X_1^V(p)$ and $X_1^V(q)$ have linearly independent columns in the same places, and hence have the same RREF.

So letting $\mathcal C_i(x)$ denote the kernel vectors obtained on input $x$ in the $i$th iteration, we have $\mathcal C_1(p) = \mathcal C_1(g\cdot p)$.  As well, let $L_i(p)$ denote the set (produced in step (11) of the algorithm) containing the linearly independent columns of $X_i^V(p)$. Then we have $L_1(p) = \{ \sigma_p^*(x_{j_1}),\ldots, \sigma_p^*(x_{j_r})\}$ and  $L_1(g\cdot p) = \{ \sigma_{g\cdot p}^*(x_{j_1}),\ldots, \sigma_{g\cdot p}^*(x_{j_r})\}$ for the same indices $j_1, \ldots, j_s$.

Proceed by induction on $i$: we may assume $X_i^V(p)$ and $X_i^V(q)$ have the same RREF and hence $\mathcal C_i(p) = \mathcal C_i(q)$. We may also assume the columns of $X_i^V(p)$ and $X_i^V(q)$ represent the images of the same set of monomials $\{x^{I_1}, \ldots, x^{I_s}\}$, for multi-indicies $I_j$.  Then the lists $V_{i+1}(p)$ and $V_{i+1}(q)$ also represent the images of the same monomials under $\sigma_p^*$ and $\sigma_q^*$, respectively.  Claim again that $X_{i+1}^V(p)$ and $X_{i+1}^V(q)$ have the same RREF.  By the induction hypothesis, the two matrices have the same basis for their domain, and the kernel of each must span $I(G\cdot p)_{i+1}$.  These facts prove the claim, as in the base case.  Thus $\mathcal C_{i+1}(p) = \mathcal C_{i+1}(q)$, and the functions in $\mathcal C$ are invariant.

To show the functions in $\mathcal C$ separate orbits, choose $p,q \in \mathbb A^n$ such that the functions in $\mathcal C$ take the same values at both points.  In particular, $\mathcal C_1(p) = \mathcal C_1(q)$, so $X_1(p)$ and $X_1(q)$ have the same canonical kernel.  As above, it follows that $X_1(p)$ and $X_1(q)$ have the same RREF.  Two facts emerge.  Crucially, the kernels of $\sigma^*_p$ and $\sigma^*_q$ have the same canonical $k$-basis for their subspaces of degree-1 elements, because the matrices $X_1^V(p)$ and $X_1^V(q)$ assume the same basis for the domain space $k[x_1, \ldots, x_n]_1$, namely, $x_1, \ldots, x_n$.  We wish to prove this for all degrees $i$.

What is more, $L_1(p) = \{ \sigma_p^*(x_{j_1}),\ldots, \sigma_p^*(x_{j_s})\}$ and  $L_1(q) = \{ \sigma_{q}^*(x_{j_1}),\ldots, \sigma_{q}^*(x_{j_s})\}$ for the same indices $j_1, \ldots, j_s$, because $X_1^V(p)$ and $X_1^V(q)$ have linearly independent columns in the same positions.  Thus $V_2(p)$ and $V_2(q)$ list the images of the same set of monomials $x_j x_k$ under $\sigma_p^*$ and $\sigma_q^*$, respectively.

Proceeding by induction, if $X_i^V(p)$ and $X_i^V(q)$ have the same RREF and list the images of the same monomials, then $X_{i+1}^V(p)$ and $X_{i+1}^V(q)$ also list the images of the same monomials.  By the assumption $C_{i+1}(p) = C_{i+1}(q)$, the matrices  $X_{i+1}^V(p)$ and $X_{i+1}^V(q)$ also have the same RREF.  Therefore the kernels of $\sigma^*_p$ and $\sigma^*_q$ have the same canonical $k$-basis for their degree-$i$ subspaces, completing the induction.  In particular, the same ideal $(f_1, \ldots, f_s)$ defines $\overline{G\cdot p}$ and $\overline{G \cdot q}$.  Since $G$ is an algebraic group, it follows $G\cdot p = G\cdot q$, completing the proof.
\end{proof}

%----------------------------------------------------------------------------

\subsection{Complexity Bounds}

The bookkeeping that follows confirms that the complexity of the orbit separating algorithm is polynomial in $n$ and $N$.  First, the degree bound $N^r M^{\ell-m}$ for a generating set of $\fq$ requires that we compute products of $N^r M^{\ell-m}$ degree-$N$ polynomials $f_i$ in $k[z_1,\ldots, z_\ell]$, for $i = 1, \ldots, N^r M^{\ell-m}$.  To this end, compute the monomials in the $z_j$ up to degree $N\cdot N^{r}M^{\ell-m}$, with total complexity $O(N^{\ell(r+1)}M^{\ell(\ell-m)})$.  Then multiply $f_1f_2 \cdots f_i$ and $f_{i+1}$ to obtain an implicit straight-line program for the product of $i+1$ distinct degree-$N$ polynomials in $k[z_1, \ldots, z_\ell]$, with complexity $O(2^{2\ell-2} i^{2\ell} N^{2\ell})$.  For details of polynomial multiplication, see Chapter 2 of \cite{Complexity}.
%Indeed, when $\ell=1$, we need not convert to the univariate case, and the multiplication is $O((iN)^{2\ell})$.

Next consider the sizes of matrices in the algorithm.  Recall that for large $d$, $H(d) \leq O(d^m)$.  Hence in iteration $i$, the matrix $X_i$ has
\[ k_i = O\left(((i-1)N)^m +n\cdot [((i-1)N)^m - ((i-2)N)^m]\right)\]

\noindent columns from $V_i$, has $\vert B(iN)\vert$ additional columns, and has $(iN)^\ell$ rows corresponding to the monomials in $k[z_1,\ldots, z_\ell]_{\leq iN}$.  Of course, $\vert B(iN)\vert = O((iN)^\ell)$, so the number of rows of $X_i$ is $O((iN)^\ell)$, and the number of columns is $O(n(iN)^m + (iN)^\ell) \leq O(n(iN)^\ell)$.  Now, computing the tRREF of an $s\times t$ matrix has complexity $O(st^2 + t^3)$.  Thus the computation of tRREF$(X_i)$ has complexity bounded by
\[ O\left( (iN)^\ell\cdot n^2(iN)^{2\ell} + n^3(iN)^{3\ell} \right) = O\left(n^3 i^{3\ell} N^{3\ell}\right).\]

\noindent The above count of the columns of $X_i$ also yields that the computation of the kernel of tRREF$(X_i)$ has complexity $O(n^2 i^{2\ell} N^{2\ell})$

In collecting the independent elements of $V_i$ in step (10), the input to the procedure $\Sigma$ is a $k_i \times (iN)^\ell$ matrix, where
\[ k_i  =  O\left(((i-1)N)^m +n\cdot \left[((i-1)N)^m - ((i-2)N)^m\right]\right) \leq O(n(iN)^m).\]

\noindent On an $s\times t$ matrix, $\Sigma$ has complexity $O(s^2 t)$, whence step (10) has complexity $\leq O(n^2 (iN)^{2m} \cdot (iN)^\ell)$.

Finally, the polynomial multiplications $f_1 \cdots f_i$ proceed through $i = N^rM^{\ell-m}$, with $n$ multiplications for each $i$.  Their total complexity is
\[ O\left(2^{2\ell-2} n (N^r M^{\ell-m})^{2\ell+1}N^{2\ell}\right) = O\left(2^{\ell-1}n N^{2\ell(r +1) + r}M^{(\ell-m)(2\ell+1)}\right).\]

Of the other computations, the programs for the tRREF have the highest cost.  Summing their complexity from $i=1$ to the degree bound, $N^r M^{\ell-m}$, yields the following:
\[ O\left( n^3 (N^r M^{\ell-m})^{3\ell+1}N^{3\ell}\right) = O\left( n^3 N^{3\ell(r+1) + r} M^{(\ell-m)(3\ell+1)}\right),\]

\noindent where, again, $N$ is the maximum polynomial degree of the representation, $M$ is a degree bound for a generating set of $\mathbb I(G) \subset k[z_1, \ldots, z_\ell]$, and under this embedding $G$ has dimension $m$.  Since the embedding $G \hookrightarrow \mathbb A^\ell$ is fixed, we omit the constant power of $M$ from the asymptotic complexity.

Finally, to bound the number of relations that the algorithm computes, we sum the column count $O(n(iN)^\ell)$ of the matrices $X_i$ over all iterations $i$, and obtain
\[ O\left(n N^{\ell(r+1)+r} M^{(\ell-m)(\ell+1)}\right) \]

\noindent polynomials generating the ideal $\fq$.  In iteration $i$, such a polynomial has $k_i \leq O(n(iN)^m)$ terms, giving a bound for the number of constructible functions that the algorithm computes:
\[ O \left(n^2 N^{(\ell+m+1)(r+1)}M^{(\ell-m)(\ell+m+1)}\right).\]

\noindent Omitting the powers of $M$, the main theorem follows.

%----------------------------------------------------------------------------

\section{Conclusion}

Given any representation of a fixed algebraic group, the algorithm writes down invariant, constructible functions that separate the orbits of the group action.  What is more, there are polynomial bounds, in the parameters of the representation, for the number of the functions and their total length as straight line programs.  These bounds describe the complexity of the problem of determining if two points lie in the same orbit, by essentially counting the number of necessary algebraic operations to perform on the coordinates of the points.  Additionally, it emerges that to separate orbits, the ``quasi-inverse" is a sufficient generalization of the ring of polynomial functions.

\section{Acknowledgments}

I would like to thank my adviser, Harm Derksen, for his indispensable ideas and guidance.

%----------------------------------------------------------------------------

\end{document}